\documentclass[12pt,reqno]{amsproc}
\usepackage{amsfonts}
\usepackage{amssymb}
\usepackage{graphicx}%

\newtheorem{theorem}{Theorem}[section]

\newtheorem{claim}{Claim}[section]

\newtheorem{corollary} {Corollary}[section]

\newtheorem{definition} {Definition}[section]

\newtheorem{lemma} {Lemma}[section]

\newtheorem{proposition} {Proposition}[section]
\newtheorem{remark} {Remark}[section]

\begin{document}

\centerline{\Large A Note on Approximately Divisible C$^*$-algebras}

\vspace{1cm}

\centerline{Weihua Li \hspace{2cm} and \hspace{2cm} Junhao Shen}
\vspace{0.3cm}

 \centerline{Mathematics Department, University of New
Hampshire, Durham, NH 03824}

\vspace{0.3cm}

\centerline{Email: whli@cisunix.unh.edu \hspace{1cm} and
\hspace{1cm} jog2@cisunix.unh.edu}

\vspace{0.3cm} \noindent\textbf {Abstract: } Let $\mathcal  A$ be a
separable, unital, approximately divisible C$^*$-algebra. We show that
$\mathcal  A$ is   generated by two self-adjoint elements and the topological free entropy
dimension of any finite generating set of $\mathcal  A$ is less than or equal to 1. In
addition, we show that the similarity degree of $\mathcal
A$ is at most 5. Thus an  approximately divisible
C$^*$-algebra has an
affirmative answer to Kadison's similarity problem.\\

\vspace{0.2cm}



\vspace{0.2cm}

\section{Introduction}
The class of approximately divisible C$^*$-algebras was introduced
by B. Blackadar, A. Kumjian and M. R$\o$rdam in \cite{b}, where they
constructed a large class of simple C$^*$-algebras having trivial
non-stable K-theory. They showed that the class of approximately
divisible C$^*$-algebras contains all simple unital AF-algebras and
most of the simple unital AH-algebras with real rank 0, and every
nonrational noncommutative torus.

The theory of free entropy and free entropy dimension was developed
by D. Voiculescu in the 1990's. It has been a very powerful tool in
the recent study of finite von Neumann algebras. In \cite{Vo3},  D. Voiculescu  introduced the notion of topological
free entropy dimension of elements in a unital C$^*$-algebra
as an analogue of
free entropy dimension in the context of C$^*$-algebra.  Recently, D. Hadwin and J. Shen
\cite{Don-shen 2} obtained some interesting results on topological free
entropy dimensions of unital C$^*$-algebras, which includes the irrational rotation C$^*$-algebras,
 UHF algebras and minimal tensor products of reduced free group C$^*$-algebras.
 Thus it will be interesting to consider the topological free entropy
dimensions  for larger class  of unital C$^*$-algebras. One motivation  of the paper is to calculate the topological free entropy
dimensions in the  approximately divisible unital C$^*$-algebras.

Note that Voiculescu's topological free entropy dimension is
defined for the finitely generated C$^*$-algebras. Therefore it is
natural to consider the generator problem for approximately
divisible unital C$^*$-algebras before we carry out the calculation
of the topological free entropy for approximately divisible unital
C$^*$-algebras. In fact the generator problem for C$^*$-algebras and
the one for von Neumann algebras
 have been studied by many people and many
 results have been obtained. For example, C. Olsen and W. Zame \cite{olsen} showed that if $\mathcal  A$
 is a unital separable C$^*$-algebra and $\mathcal  B$ is a UHF algebra,
 then ${\mathcal  A}\otimes {\mathcal  B}$ is generated by two self-adjoint elements in  ${\mathcal  A}\otimes {\mathcal  B}$.
 It is clear that such ${\mathcal  A}\otimes {\mathcal  B}$ is approximately divisible. In the paper we obtain the following result
   (see Theorem 3.1), which is an extension of C. Olsen and W. Zame's result in \cite{olsen}.

 \vspace{0.2cm}

 \noindent {\textbf{Theorem: }If $\mathcal  A$ is a unital separable approximately divisible
C$^*$-algebra, then $\mathcal  A$ is  generated by two self-adjoint
elements in  ${\mathcal  A}$, i.e., $\mathcal  A$ is singly
generated.

 \vspace{0.2cm}

Next, we develop the techniques from \cite{Don-shen 2} and compute the topological free entropy dimension
of any finite family of self-adjoint generators of a unital separable approximately divisible
C$^*$-algebra. More specifically, we obtain the following result (see Theorem 4.3).

 \vspace{0.2cm}

 \noindent {\textbf{Theorem: }\label{theorem, lower bound}
Let $\mathcal{A}$ be a unital separable approximately divisible
C$^{*}$-algebra. If $\mathcal  A$ has approximation property, then $$\delta
_{top}(x_1,\ldots, x_n)= 1,$$ where $x_1,\ldots,x_n$ is any family of self-adjoint generators of $\mathcal A$.
\vspace{0.2cm}

In the last part of the paper, we study the Kadison's similarity problem for approximately divisible
C$^*$-algebras. In \cite{Kadison1},
  R. Kadison formulated his famous similarity problem for a C$^*$-algebra $\mathcal A$, which asks the following question:
Let $\pi: {\mathcal  A}\rightarrow {\mathcal  B(\mathcal H)}$
($\mathcal H$ is a Hilbert space) be a unital bounded homomorphism.
Is $\pi$ similar to a
*-homomorphism, that is, there exists an invertible operator
$S\in{\mathcal  B(\mathcal H)}$ such that $S^{-1}\pi(\cdot)S$ is a
*-homomorphism?

 G. Pisier \cite{Pisier 1} introduced a powerful concept,
similarity degree of a C$^*$-algebra,   to determine
whether Kadison's similarity problem for a
C$^*$-algebra  has an affirmative answer. In fact, he showed that a C$^*$-algebra $\mathcal  A$ has an
affirmative answer to Kadison's similarity problem if and only if
the similarity degree of $\mathcal  A$, $d({\mathcal  A})$, is finite.
The similarity degrees of some classes of C$^*$-algebras have been known, which we list as below.
\begin{enumerate}
\item  $\mathcal  A$ is nuclear if and only if $d({\mathcal  A})=2$
( \cite{Bunce1}, \cite{Christensen1}, \cite{Pisier2});
\item if ${\mathcal  A}={\mathcal  B(\mathcal H)}$, then $d({\mathcal  A})=3$
(\cite{Pisier3});
\item $d({\mathcal  A}\otimes {\mathcal  K(\mathcal H)})\leq 3$ for any C$^*$-algebra
$\mathcal  A$ (\cite{Haagerup 1}, \cite{Pisier4});
\item if $\mathcal  M$ is a factor of type II$_1$ with property $\Gamma$,
then $d({\mathcal  M})=3$ (\cite{Christensen2}).
\end{enumerate}
The last result (see Theorem 5.1) we obtain in the paper is the calculation of similarity degree of approximately divisible
C$^*$-algebras.

\vspace{0.2cm} \noindent {\textbf{Theorem: }}  If $\mathcal  A$ is a
unital separable approximately divisible C$^*$-algebra, then
$$d({\mathcal  A})\leq 5.$$ As   a
corollary, an  approximately divisible C$^*$-algebra has an
affirmative answer to Kadison's similarity problem.

The paper has five sections. In section 2, we recall the definition of  approximately divisible
C$^*$-algebra. The generator problem for an  approximately divisible
C$^*$-algebra is considered in section 3. The computation of topological free entropy dimension in an  approximately divisible
C$^*$-algebra is carried out in section 4. In section 5, we consider the similarity degree of an  approximately divisible
C$^*$-algebra.

\section{Notation and preliminaries}

In this section, we will introduce some notation that will be needed
later and recall the definition of approximately divisible
C$^*$-algebra introduced by B. Blackadar, A. Kumjian and M.
R$\o$rdam \cite{b}.

Let ${\mathcal  M}_k({\Bbb C})$ be the $k\times k$ full matrix
algebra with entries in $\Bbb C$, and ${\mathcal  M}_k^{s.a}({\Bbb
C})$ be the subalgebra of ${\mathcal  M}_k({\Bbb C})$ consisting of
all self-adjoint matrices of ${\mathcal  M}_k({\Bbb C})$. Let
${\mathcal  U}_k$ be the group of all unitary matrices in ${\mathcal
M}_k({\Bbb C})$.  Let ${\mathcal  M}_k({\Bbb C})^n$ denote the
direct sum of $n$ copies of ${\mathcal  M}_k({\Bbb C})$. Let
$({\mathcal  M}_k^{s.a}({\Bbb C}))^n$ be the direct sum of $n$
copies of ${\mathcal  M}_k^{s.a}({\Bbb C})$.

The following lemma is a well-known fact.
\begin{lemma}
Suppose $\mathcal  B$ is a finite-dimensional C$^*$-algebra. Then there
exist   positive integers $r$ and  $k_1, \ldots, k_r$ such that
 $${\mathcal  B}\cong {\mathcal  M}_{k_1}({\Bbb C})\oplus \cdots \oplus {\mathcal  M}_{k_r}({\Bbb
 C}).$$
\end{lemma}

\begin{definition}
Suppose
 $${\mathcal  B}\cong {\mathcal  M}_{k_1}({\Bbb C})\oplus \cdots \oplus {\mathcal  M}_{k_r}({\Bbb
 C}) $$ is a finite-dimensional C$^*$-algebra for some positive integers $r, k_1,\ldots, k_r$.
Define the {\it rank} of $\mathcal  B$ to be $$\mbox{Rank}({\mathcal  B})=k_1+\cdots
 +k_r,$$
  the {\it subrank} of $\mathcal  B$ to be $$\mbox{SubRank}({\mathcal  B})
 =\min\{k_1, \ldots, k_r\}.$$
\end{definition}

 The
following definition is Definition 1.2 in \cite{b}.
\begin{definition}
A separable unital C$^{*}$-algebra $\mathcal{A}$ with the unit
$I_{\mathcal A}$ is approximately divisible if, for every
$x_{1},\ldots,x_{n}\in\mathcal{A}$ and $\varepsilon>0$, there is a
finite-dimensional C$^{*}$-subalgebra $\mathcal{B}$ of $\mathcal{A}$
 such that

(1) $I_{\mathcal  A}\in {\mathcal  B}$;

(2) $\mbox{SubRank}({\mathcal  B})\geq 2$;

 (3) $\|x_{i}y-yx_{i}\|<\varepsilon$
for $i=1,\ldots, n$ and all $y\in {\mathcal  B}$ with $\|y\|\leq 1$.
\end{definition}

The following proposition is also taken from Theorem 1.3 and
Corollary 2.10 in \cite{b}.

\begin{proposition}
\label{proposition, ad form}(\cite{b}) Let $\mathcal{A}$ be a unital
separable approximately divisible C$^{\ast}$-algebra with the unit
$I_{\mathcal  A}$. Then there exists an increasing sequence
$\{{\mathcal A}_m\}_{m=1}^{\infty}$ of subalgebras of $\mathcal  A$
such that

(1) $\mathcal{A}=\overline{\cup_m{\mathcal{A}}_{m}}^{\|\cdot\|}$,

(2) for any positive integer $m$, ${\mathcal{A}}_{m}^{\prime}%
\cap{\mathcal{A}}_{m+1}$ contains a  finite-dimensional
C$^{\ast}$-subalgebra $\mathcal  B$ with $I_{\mathcal  A}\in {\mathcal  B}$ and
$\mbox{SubRank}({\mathcal  B})\geq 2$,

(3) for any positive integers $m$ and $k$, there is a finite-dimensional C$^{\ast}%
$-subalgebra $\mathcal{B}$ of
${\mathcal{A}}_{m}^{\prime}\cap\mathcal{A}$ with $I_{\mathcal  A}\in
{\mathcal  B}$ and $\mbox{SubRank}({\mathcal  B})\geq k$.
\end{proposition}

\section{Generator problem of approximately divisible C$^*$-algebras}

In this section we prove that every unital separable approximately
divisible C$^*$-algebra is singly generated, i.e., generated by two
self-adjoint elements.

\begin{theorem}\label{theorem, singly generated} If $\mathcal  A$ is a unital separable approximately divisible
C$^*$-algebra, then $\mathcal  A$ is singly generated.
\end{theorem}
\begin{proof} Since $\mathcal  A$ is separable, there exists a sequence of
self-adjoint elements $\{x_i\}_{i=1}^{\infty}\subset \mathcal A$
that generate $\mathcal  A$ as a C$^*$-algebra.
\begin{claim}\label{claim,singly,exists sequence}There exists a sequence of finite-dimensional subalgebras
$\{{\mathcal  B}_n\}_{n=1}^{\infty}$ of $\mathcal  A$ so that the following
hold:

(1) $\forall \ n \in \Bbb N$,   $I_{\mathcal  A}\in {\mathcal B}_n$,
where $I_{\mathcal  A}$ is the unit of $\mathcal  A$;

(2) $\mbox{SubRank}({\mathcal  B}_1)\geq 3$, and for any $n\geq 2$,
$$\mbox{SubRank}({\mathcal  B}_n)\geq n\cdot (\mbox{Rank}({\mathcal  B}_1))^2\cdots
(\mbox{Rank}({\mathcal  B}_{n-1}))^2+3;$$

(3) if $n\neq m$, then ${\mathcal  B}_n$ commutes with ${\mathcal  B}_m$;

(4) for any   $n\in\Bbb N$, $$dist(x_p, {\mathcal
B}_n^{\prime}\cap {\mathcal  A})< 2^{-n},\hspace{0.5cm} \forall 1\leq
p\leq n,$$ where $dist(x_p, {\mathcal  B}_n'\cap {\mathcal
A})=\inf\{\|x_p-y\|: y\in {\mathcal  B}_n'\cap {\mathcal  A}\}$.
\end{claim}
{\it Proof of the claim}.
It follows from  Proposition \ref{proposition, ad
form} that
 there exists an increasing sequence $\{{\mathcal
A}_m\}_{m=1}^{\infty}$ of subalgebras of $\mathcal  A$ such that
\begin{enumerate}
\item [(a)]
  $\mathcal{A}=\overline{\cup_m{\mathcal{A}}_{m}}^{\|\cdot\|}$,
\item [(b)] for any positive integer $m$, ${\mathcal{A}}_{m}^{\prime}%
\cap{\mathcal{A}}_{m+1}$ contains a  finite-dimensional
C$^{\ast}$-subalgebra $\mathcal  B$ with $I_{\mathcal  A}\in {\mathcal  B}$ and
$\mbox{SubRank}({\mathcal  B})\geq 2$,
\item [(c)] for any positive integers $m$ and $k$, there is a finite-dimensional C$^{\ast}%
$-subalgebra $\mathcal{B}$ of
${\mathcal{A}}_{m}^{\prime}\cap\mathcal{A}$ with $I_{\mathcal  A}\in
{\mathcal  B}$ and $\mbox{SubRank}({\mathcal  B})\geq k$.
\end{enumerate}
Instead of proving Claim 3.1 directly, we will prove  a stronger result by replacing the statement
(3) in Claim 3.1 with the following  one:

(3$'$) there exist two increasing sequences $\{s_n\}_{n=1}^{\infty}$
and $\{t_n\}_{n=1}^{\infty}$ of positive integers such that, for any
  $n\in \Bbb N$, $s_n\leq t_n\leq s_{n+1}$ and ${\mathcal
B}_n\subseteq {\mathcal  A}_{s_n}'\cap {\mathcal  A}_{t_n}$.

\vspace{0.2cm}

\noindent We prove this stronger claim by using the induction on $n$.

\textbf{Base step:} Note that ${\mathcal
A}=\overline{\cup_m{\mathcal  A}_m}^{\|\cdot\|}$. For $x_1\in
\mathcal A$, there are a positive integer $s_1$ and a self-adjoint
element
  $y_1^{(1)}\in{\mathcal  A}_{s_1}$  such that
$\|x_1-y_1^{(1)}\|<\frac{1}{2}$. By the restriction (b) on the
subalgebras $\{{\mathcal  A}_m\}_{m=1}^\infty$, we know that there
exist two finite-dimensional subalgebras ${\mathcal C}_{s_1+1},
{\mathcal  C}_{s_1+2}$ in $\mathcal  A$ such that,

(i) $I_{\mathcal  A} \in {\mathcal  C}_{s_1+1}$ and $I_{\mathcal  A} \in {\mathcal  C}_{s_1+2}$;

(ii) ${\mathcal  C}_{s_1+1}\subseteq {\mathcal  A}_{s_1}'\cap
{\mathcal  A}_{s_1+1}$ and ${\mathcal  C}_{s_1+2}\subseteq {\mathcal
A}_{s_1+1}'\cap {\mathcal A}_{s_1+2}$;

(iii) $\mbox{SubRank}({\mathcal  C}_{s_1+1})$ and $\mbox{SubRank}({\mathcal
C}_{s_1+2})$ are at least 2.
 \newline Let $t_1=s_1+2$, ${\mathcal  B}_1=C^*({\mathcal  C}_{s_1+1},{\mathcal  C}_{t_1} )$ the $^*$-subalgebra generated by
 ${\mathcal  C}_{s_1+1}$ and ${\mathcal  C}_{t_1} $ in $\mathcal A$. Then $\mbox{SubRank}({\mathcal
B}_1)\geq 3$ and ${\mathcal  B}_1\subseteq {\mathcal
A}_{s_1}'\cap{\mathcal A}_{t_1}$.

\textbf{Inductive step:} Now suppose the stronger claim is true when
$n\leq k-1$, i.e., there exists a family  of finite-dimensional
C$^*$-algebras $\{{\mathcal  B}_n\}_{n=1}^{k-1}$ of $\mathcal  A$,
and two increasing sequences of positive integers
$\{s_n\}_{n=1}^{k-1}$ and $\{t_n\}_{n=1}^{k-1}$ that satisfy  (1),
(2), (3$'$) and (4).

For $x_1, \ldots, x_{k}$ in $\mathcal  A$, from the restriction (a)
on $\{{\mathcal  A}_m\}_{m=1}^\infty\subseteq \mathcal A$, we know
that there are a positive integer $s_{k}$ with $s_{k}\geq t_{k-1}$
and self-adjoint elements $y_1^{(k)}, \ldots, y_{k}^{(k)}$ in
${\mathcal  A}_{s_{k}}$ such that $\|x_i-y_i^{(k)}\|<2^{-k}$ for
$1\leq i\leq k$. From the restriction (b) on $\{{\mathcal
A}_m\}_{m=1}^\infty\subseteq \mathcal A$,
 there exists a family $\{{\mathcal
C}_{s_{k}+1},{\mathcal  C}_{s_{k}+2},\ldots\}$ of finite-dimensional
subalgebras in $\mathcal  A$ such that,

(i) $I_{\mathcal  A} \in {\mathcal  C}_{s_{k}+i}$, $\ \forall \ i\ge 1$;

(ii) ${\mathcal  C}_{s_{k}+i}\subseteq {\mathcal  A}_{s_{k}+i-1}'\cap {\mathcal
A}_{s_{k}+i}$, $\ \forall \ i\ge 1$;

(iii) $\mbox{SubRank}({\mathcal  C}_{s_{k}+i})\geq 2$,  $\ \forall \
i\ge 1$.\newline By (ii), we know $\{{\mathcal
C}_{s_{k}+1},{\mathcal  C}_{s_{k}+2},\ldots\}$ is a commuting
sequence of subalgebras of $\mathcal A$. Combining with (iii), we
get that there is a positive integer $t_{k}$ such that
$$\mbox{SubRank}\left(C^*({\mathcal  C}_{s_{k}+1}, \ldots, {\mathcal
C}_{t_{k}})\right)\geq k\cdot(\mbox{Rank}({\mathcal  B}_1))^2\cdots
(\mbox{Rank}({\mathcal  B}_{k-1}))^2+3,$$ where $ C^*({\mathcal  C}_{s_{k}+1}, \ldots, {\mathcal
C}_{t_{k}}) $ is the C$^*$-subalgebra generated by ${\mathcal  C}_{s_{k}+1}, \ldots, {\mathcal
C}_{t_{k}}$ in $\mathcal A$. Moreover, $I_{\mathcal  A} \in  C^*({\mathcal
C}_{s_{k}+1}, \ldots, {\mathcal  C}_{t_{k}})$ is a finite-dimensional
C$^*$-subalgebra in ${\mathcal  A}_{s_{k}}'\cap{\mathcal  A}_{t_{k}}.$ Let
$${\mathcal  B}_{k}=C^*({\mathcal  C}_{s_{k}+1}, \ldots, {\mathcal  C}_{t_{k}})$$ and it is not hard to check that
$\mathcal B_1,\ldots, \mathcal B_k$ satisfy the conditions (1), (2), (3$'$) and (4) in the stronger claim.  {\it This completes
the proof of the claim.}

\vspace{0.8cm}
Let $\{\mathcal B_n\}_{n=1}^\infty$ be as in Claim 3.1.
 For any positive integer $n$, since ${\mathcal  B}_n$ is a finite-dimensional C$^*$-algebra,
 there exist positive integers $ r_n$ and $k_1^{(n)}, \ldots,
 k_{r_n}^{(n)}$ such that
 $${\mathcal  B}_n\cong{\mathcal  M}_{k_1^{(n)}}(\Bbb C)\oplus \cdots \oplus {\mathcal  M}_{k_{r_n}^{(n)}}({\Bbb
 C}).$$ Let $\{e_{ij}^{(n,s)}: 1\leq i,j\leq k_s^{(n)}\}$ be the canonical system  of matrix units for $\mathcal M_{k_s^{(n)}}$.
  If there is no confusion arising, we can further assume that $\{e_{ij}^{(n,s)}: 1\leq i,j\leq k_s^{(n)}, 1\leq s\leq r_{n}\}$
  consists a system of matrix units of $\mathcal B_n$.
  Note that ${\mathcal  B}_n$ contains the unit $I_{\mathcal  A}$ of ${\mathcal  A}$, so
  $$\sum_{s=1}^{r_n}\sum_{i=1}^{k_s^{(n)}}e_{ii}^{(n,s)}=I_{\mathcal
  A}.$$

   Define \begin{equation} p_n=\sum_{s=1}^{r_n}e_{k_s^{(n)},k_s^{(n)}}^{(n,s)}\ \ \ \  { for } \
 n\geq 1 .\end{equation} Then $p_n$ is a projection of ${\mathcal  B}_n$.
 It is clear that \begin{equation}p_ne_{11}^{(n,s)}=0 \hspace{0.5cm} {for}\  1\leq s\leq r_n.\end{equation}

\begin{claim}\label{claim, elements is in it} Let $\{x_n\}_{n=1}^\infty$, $\{{\mathcal  B}_n\}_{n=1}^{\infty}$,
 $\{r_n\}_{n=1}^{\infty}$ and $\{p_n\}_{n=1}^{\infty}$ be defined as
above. For any positive integer $n$, there exists $z_n=z_n^*\in\mathcal
A$ with $\|z_n\|=2^{-(r_1+\cdots +r_n+1)}$ so that

(i) $(I_{\mathcal  A}-p_n)p_{n-1}\cdots p_1\cdot z_n\cdot p_1\cdots
p_{n-1}(I_{\mathcal  A}-p_n)=z_n,$

(ii) $dist(x_j, C^*({\mathcal  B}_1, \ldots, {\mathcal  B}_{n}, z_n))< 2^{-n}$
for $1\leq j\leq n$, where $ C^*({\mathcal  B}_1, \ldots, {\mathcal  B}_{n}, z_n)$ is the C$^*$-subalgebra generated
by ${\mathcal  B}_1, \ldots, {\mathcal  B}_{n}, z_n$ in $\mathcal A$.
\end{claim}

{\it Proof of the claim}. By Claim \ref{claim,singly,exists
sequence}, for any positive integer $n$ and $x_1,\ldots, x_n$, we know  $$dist(x_j, {\mathcal  B}_n'\cap
{\mathcal  A})< 2^{-n} \ \ \ \ \ for \ \ 1\leq j\leq n. $$  Thus there exist
self-adjoint elements $y_1^{(n)}, \ldots, y_n^{(n)}$ in $\mathcal  A$
that commute with ${\mathcal  B}_n$ and $$\|x_j-y_j^{(n)}\|< 2^{-n} \ \ \ \  for \ \ 1\leq j\leq n. $$

 Let
\begin{equation}z_1=\frac 1{ 2^{1+r_1}} \cdot \frac{\sum_{s=1}^{r_1}e_{22}^{(1,s)}y_1^{(1)}}{  \|
\sum_{s=1}^{r_1}e_{22}^{(1,s)}y_1^{(1)} \|}.\end{equation} With $\mbox{SubRank}({\mathcal
B}_1)\geq 3, $ we have
$$(I_{\mathcal  A}-p_1)\cdot z_1\cdot (I_{\mathcal  A}-p_1)=z_1.$$
By the equation (3) and the fact that $y_1^{(1)}$ commutes with
$\mathcal B_1$, we know
$$y_1^{(1)}=\left( { 2^{1+r_1}} \cdot  {  \| \sum_{s=1}^{r_1}e_{22}^{(1,s)}y_1^{(1)} \|}\right) \cdot
\left (\sum_{s=1}^{r_1}\sum_{i=1}^{k_s^{(1)}} e_{i,2}^{(1,s)}\cdot  z_1\cdot e_{2,i}^{(1,s)}\right ).$$
Thus
we know that $y_1^{(1)}$ is in the C$^*$-algebra generated by ${\mathcal
B}_1$ and $z_1$, whence $$dist(x_1, C^*({\mathcal  B}_1, z_1))\le dist (x_1, y_1^{(1)})< 2^{-1}.$$

Now let us construct $z_n$ for any positive integer $n\geq 2$. Let
\begin{eqnarray*}\Delta_{n-1}&=&\{(i_1,s_1)\times(j_1,
t_1)\times\cdots\times
(i_{n-1},s_{n-1})\times (j_{n-1},t_{n-1}):\\
&& 1\leq i_1\leq k_{s_1}^{(1)},
 1\leq j_1\leq k_{t_1}^{(1)},1\leq s_1,t_1\leq r_1,\\
 &&\cdots, 1\leq i_{n-1}\leq k_{s_{n-1}}^{(n-1)},
 1\leq j_{n-1}\leq k_{t_{n-1}}^{(n-1)},1\leq s_{n-1},t_{n-1}\leq r_{n-1}\}.\end{eqnarray*}
 It is not hard to check that the cardinality of the set $\Delta_{n-1}$ satisfies
 $$Card (\Delta_{n-1})= \prod_{i=1}^{n-1} (Rank(\mathcal B_i))^2 .$$    Hence, for any $1\leq j\leq n$, there is a one-to-one mapping
$f_j^{(n)}$  from the index set $ \Delta_{n-1}$ onto the set $$\{i\in \Bbb N\ | (j-1)\cdot Card (\Delta_{n-1})+2\le i\le
j\cdot Card (\Delta_{n-1}) +1 \}.$$

For any index  $$\alpha=(i_1,s_1)\times(j_1, t_1)\times\cdots\times
(i_{n-1},s_{n-1})\times (j_{n-1},t_{n-1})\in\Delta_{n-1}$$ and any
$1\leq j\leq n$, we should define
\begin{equation}\alpha(y_j^{(n)})=e_{k_{s_{n-1}}^{(n-1)},i_{n-1}}^{(n-1, s_{n-1})}\cdots
e_{k_{s_1}^{(1)},i_1}^{(1,s_1)}\cdot y_j^{(n)}\cdot
e_{j_1,k_{t_1}^{(1)}}^{(1,t_1)} \cdots
e_{j_{n-1},k_{t_{n-1}}^{(n-1)}}^{(n-1, t_{n-1})}\in \mathcal A.\end{equation} By Claim 3.1, we know  that
$\mbox{SubRank}({\mathcal  B}_n)\geq n \cdot Card(\Delta_{n-1})+3$. It follows that
\begin{eqnarray}z_{n}&=&c_n\cdot
 \sum_{s=1}^{r_n}\sum_{j=1}^{n}\sum_{\alpha\in\Delta_{n-1}}\left(
e_{f_j^{(n)}(\alpha), f_j^{(n)}(\alpha)+1}^{(n,s)}\cdot
\alpha(y_j^{(n)})\right. \left.+(e_{f_j^{(n)} (\alpha),
f_j^{(n)}(\alpha)+1}^{(n,s)}\cdot
\alpha(y_j^{(n)}))^*\right)\end{eqnarray}  is well defined and
contained in $\mathcal A$, where   $c_n$ is a constant such that
\begin{equation}\|z_{n}\|=2^{-(r_1+\cdots +r_n+1)}.\end{equation}
From the construction of $z_n$, it follows that $z_n=z_n^*$ and
$$ z_{n}=(I_{\mathcal  A}-p_n)\cdot p_{n-1}\cdots p_1\cdot z_n\cdot p_1\cdots
p_{n-1}\cdot(I_{\mathcal  A}-p_n).$$

To prove  $dist(x_j, C^*({\mathcal  B}_1, \ldots, {\mathcal  B}_{n}, z_n))<
2^{-n}$ for $1\leq j\leq n$, it is sufficient to prove that
$\{y_1^{(n)}, \ldots, y_n^{(n)}\}\subseteq C^*({\mathcal  B}_1, \ldots,
{\mathcal  B}_{n}, z_n)$. Because ${\mathcal  B}_n$ commutes with $y_1^{(n)},
\ldots, y_n^{(n)}$,   for any $\alpha\in\Delta_{n-1}$ and $1\leq
j\leq n$, from the equation (5) it follows that
$$\alpha(y_j^{(n)})=\sum_{s=1}^{r_n}\sum_{i=1}^{k_s^{(n)}} e_{i,f_j^{(n)}(\alpha)}^{(n,s)}\cdot \left(\frac{1}{c_n}z_n\right)
\cdot e_{f_j^{(n)}(\alpha)+1,i}^{(n,s)}.$$
This implies that $\alpha(y_j^{(n)})\in C^*({\mathcal  B}_1, \ldots,
{\mathcal  B}_n, z_n)$.

 Suppose
$\alpha=(i_1, s_1)\times(j_1, t_1)\times\cdots \times(i_{n-1},
s_{n-1})\times (j_{n-1}, t_{n-1})\in \Delta_{n-1}$. Again because
${\mathcal B}_1, \ldots, {\mathcal  B}_{n-1}$ are commuting, for any
$1\leq j\leq n$, from the equation (4) it follows that
\begin{eqnarray*} &&e_{i_{n-1},i_{ {n-1}} }^{(n-1, s_{n-1})} \cdots
e_{i_1,i_{1} }^{(1,s_1)} \cdot  y_j^{(n)} \cdot e_{j_{ 1}
,j_1}^{(1,t_1)} \cdots
e_{ {j_{n-1}} ,j_{n-1}}^{(n-1, t_{n-1})}\\
&=& e_{i_{n-1},k_{s_{n-1}}^{(n-1)}}^{(n-1, s_{n-1})}\cdots
e_{i_1,k_{s_1}^{(1)}}^{(1,s_1)}\cdot \alpha(y_j^{(n)})\cdot
e_{k_{t_1}^{(1)},j_1}^{(1,t_1)} \cdots
e_{k_{t_{n-1}}^{(n-1)},j_{n-1}}^{(n-1, t_{n-1})}\end{eqnarray*} and
$$\begin{aligned} y_j^{(n)}&=
\sum_{s_1,t_1=1}^{r_1}\sum_{i_1=1}^{k_{s_1}^{(1)}}\sum_{j_1=1}^{k_{t_1}^{(1)}}\cdots
\sum_{s_{n-1},t_{n-1}=1}^{r_{n-1}}\sum_{i_{n-1}=1}^{k_{s_{n-1}}^{(n-1)}}\sum_{j_{n-1}=1}^{k_{t_{n-1}}^{(n-1)}}\\
 &\qquad \quad \qquad \qquad \qquad \qquad e_{i_{n-1},i_{ {n-1}} }^{(n-1, s_{n-1})} \cdots
e_{i_1,i_{1} }^{(1,s_1)} \cdot  y_j^{(n)} \cdot
e_{j_{ 1} ,j_1}^{(1,t_1)} \cdots
e_{ {j_{n-1}} ,j_{n-1}}^{(n-1, t_{n-1})}.\end{aligned}$$
Thus $y_j^{(n)}\in C^*({\mathcal  B}_1, \ldots, {\mathcal  B}_n, z_n)$. Hence
$$dist(x_j, C^*({\mathcal  B}_1, \ldots, {\mathcal  B}_{n}, z_n))\le dist(x_j,y_j^{(n)})< 2^{-n}\ \ \  for
 \ \ 1\leq j\leq n. $$ {\it This completes the proof of the
claim.}\vspace{0.8cm}

Let $\{x_n\}_{n=1}^\infty$, $\{{\mathcal  B}_n\}_{n=1}^{\infty}$,
 $\{r_n\}_{n=1}^{\infty}$, $\{p_n\}_{n=1}^{\infty}$ and $\{z_n\}_{n=1}^\infty$ be  as
above. From the equation (2), the fact (i) of Claim 3.2 and the
construction of $z_n$, we can get some basic facts of $z_n$. Let us
list them below:

(F1)\ $p_nz_n=z_np_n=0,$

 (F2)\ $ z_n\cdot e_{11}^{(m,s)}=e_{11}^{(m,s)}\cdot
z_n=0 \ \ \ for\ m\leq n \ \ and \ 1\leq s\leq r_m,$

(F3)\ $ z_n\cdot z_m=0 \ \ \ for \ \ any \ \ n\neq m.$

Let $p_0=I_{\mathcal  A}$ and $r_0=0$. For any $n\geq 1$, let
\begin{equation}a_n=p_1\cdots p_{n-1}\sum_{s=1}^{r_n}2^{-r_1-\cdots
-r_{n-1}-s}\cdot e_{11}^{(n,s)}+z_n,\end{equation}
\begin{equation}b_n=2^{-2n}p_1\cdots p_{n-1}\sum_{s=1}^{r_n}\sum_{i=1}^{k_s^{(n)}-1}
(e_{i,i+1}^{(n,s)}+e_{i+1,i}^{(n,s)}).\end{equation}From the proven
facts (F1), (F2) and (F3), we have
\begin{equation} a_n\cdot a_m=0. \ \ \ \ for \ \ n \ne m\end{equation}
Combining the fact (F2), the equation (6) and the fact
$e_{11}^{(n,s)}\cdot e_{11}^{(n,s_1)}= e_{11}^{(n,s_1)}\cdot
e_{11}^{(n,s)}=0$ ($s\neq s_1$), it is clear that
\begin{equation}\|a_n\|=\max\{\{\|p_1\cdots p_{n-1}\cdot 2^{-r_1-\cdots
-r_{n-1}-s}\cdot
e_{11}^{(n,s)}\|\}_{s=1}^{r_n},\|z_n\|\}=2^{-r_1-\cdots
-r_{n}-1}\leq 2^{-n}.\end{equation}By the equation (8), we get
\begin{equation}\|b_{n}\|\leq 2\cdot
2^{-2n}\cdot\|\sum_{s=1}^{r_n}\sum_{i=1}^{k_s^{(n)}-1}e_{i,i+1}^{n,s}\|\leq
2^{-2n+1}\leq 2^{-n}.
\end{equation}
 It induces that both
$\sum_{n=1}^{\infty}a_n$ and $\sum_{n=1}^{\infty}b_n$ are all
convergent series in $\mathcal A$. Let \begin{equation}
a=\sum_{n=1}^{\infty}a_n, \hspace{1.5cm}
b=\sum_{n=1}^{\infty}b_n.\end{equation} It is clear that
$a=a^*\in\mathcal  A$ and $b=b^*\in\mathcal  A$.

\begin{claim}\label{claim, algebra is in it}Let $\{{\mathcal  B}_n\}_{n=1}^{\infty}$,
$\{z_n\}_{n=1}^{\infty}$ and $a, b$ be defined as above. Then
$$\{{\mathcal  B}_1, z_1, {\mathcal  B}_2, z_2,\ldots\}\subseteq C^*(a,b),$$ where C$^*(a,b)$ is the C$^*$-subalgebra
generated by $a$ and $b$ in $\mathcal A$.
\end{claim}

{\it Proof of the claim.} It is sufficient to prove that, for any
$n\geq 1$, $$\{{\mathcal  B}_1, \ldots, {\mathcal  B}_n, z_1, \ldots,
z_n\}\subseteq C^*(a,b).$$ We will prove it by using the induction on $n$.

First we prove $\{{\mathcal  B}_1, z_1\}\subseteq C^*(a,b)$. From equations (7), (9), (12) and part (i)
of Claim 3.2, it follows that \begin{eqnarray*}\forall \ k\in\Bbb N, \ \ (2a)^k&=&(2a_1)^k+\sum_{n=2}^\infty(2a_n)^k\\
&=& e_{11}^{(1,1)}+
\sum_{s_1=2}^{r_1}(2^{-s_1+1})^ke_{11}^{(1,s_1)}+(2z_1)^k+\sum_{n=2}^\infty(2a_n)^k.
\end{eqnarray*}
By inequalities (6), (9) and (10), we obtain that
$$
\|\sum_{s_1=2}^{r_1}(2^{-s_1+1})^ke_{11}^{(1,s_1)}+(2z_1)^k+\sum_{n=2}^\infty(2a_n)^k\|\rightarrow 0, \ \ \ as \ k\rightarrow \infty,
$$which implies
$\|(2a)^k- e_{11}^{(1,1)}\|\rightarrow 0$, as $k$ goes to $\infty$. Thus, $e_{11}^{(1,1)}\in C^*(a,b)$.
By the construction of the element $b$, it is not hard to check  $\{e_{ij}^{(1,1)}: 1\leq
i,j\leq k_1^{(1)}\}$ are contained in the C$^*$-subalgebra generated by
$e_{11}^{(1,1)}$ and $b$ in $\mathcal A$. Therefore, $\{e_{ij}^{(1,1)}: 1\leq
i,j\leq k_1^{(1)}\} \subseteq C^*(a,b)$.

Since
$$(I_{\mathcal  A}-e_{11}^{(1,1)})\cdot a\cdot (I_{\mathcal  A}-e_{11}^{(1,1)})=
\sum_{s=2}^{r_1}2^{-s}e_{11}^{(1,s)}+z_1 +\sum_{n=2}^\infty  a_n ,$$ by equation (6) and inequality (10)  we know
$$\left \|\left(4(I_{\mathcal  A}-e_{11}^{(1,1)})\cdot a\cdot (I_{\mathcal  A}-e_{11}^{(1,1)})\right)^k-
e_{11}^{(1,2)}\right \|\rightarrow 0, \ \ as \ k \rightarrow \infty.$$
This implies $e_{11}^{(1,2)}$ is in $C^*(a,b)$, whence   $\{e_{ij}^{(1,2)}:1\leq i,j\leq
k_2^{(1)}\}$  are in $C^*(a,b)$. Repeating the preceding process, we get that
 $\{e_{ij}^{(1,s)}: 1\leq i,j\leq k_s^{(1)}, 1\leq s\leq r_1\}$ and, therefore $p_1$, are contained in $C^*(a,b)$.
 By fact (i) of the Claim 3.2, we know $$(I_{\mathcal A}-p_1)z_1=z_1;$$ and, from equation (7) we
 know $$(I_{\mathcal A}-p_1)\sum_{n=2}^\infty a_n =0.$$ This induces that $z_1$ is also contained in  $C^*(a,b)$. Now we conclude that
both  ${\mathcal  B}_1$ and $z_1$ are contained in $C^*(a,b)$.

Assume that $\{{\mathcal  B}_1,\ldots, {\mathcal
B}_{n-1},z_1,\ldots, z_{n-1}\}\subseteq C^*(a,b)$. We need to prove
that $\{{\mathcal B}_n, z_n\}\subseteq C^*(a,b)$. By the equation
(2) and the construction of the elements $a$, $b$ (see equations
(7), (8), (12)), we know that
$$\left(p_1\cdots p_{n-1}\right)a=\sum_{i=n}^{\infty}a_i=p_1\cdots p_{n-1}\sum_{s=1}^{r_n}2^{-r_1-\cdots
-r_{n-1}-s}\cdot e_{11}^{(n,s)}+z_n +\sum_{i=n+1}^{\infty}a_i,
$$ and
$$\left(p_1\cdots p_{n-1}\right)b\left(p_1\cdots
p_{n-1}\right) =\sum_{i=n}^{\infty}b_i=\left (2^{-2n}p_1\cdots p_{n-1}\sum_{s=1}^{r_n}\sum_{i=1}^{k_s^{(n)}-1}
(e_{i,i+1}^{(n,s)}+e_{i+1,i}^{(n,s)})\right)+\sum_{i=n+1}^{\infty}b_i.
$$ By   equations (7), (9) and part (i) of Claim 3.2, we obtain that
$$\left \|\left(2^{r_1+\cdots +r_{n-1}+1}\right) \left (\sum_{i=n}^{\infty}a_i\right)^k-
\left(p_1\cdots p_{n-1}\right)e_{11}^{(n,1)} \right\|\rightarrow 0,
\ \ as \ \ k \rightarrow\infty.$$ Using the similar argument as in
the case $n=1$, we know that  $\{p_1\cdots
p_{n-1}e_{ij}^{(n,s)}:1\leq i,j\leq k_s^{(n)}, 1\leq s\leq r_n\}$
are contained in the C$^*$-subalgebra generated by
$\sum_{i=n}^{\infty}a_i$ and $\sum_{i=n}^{\infty}b_i$ in $\mathcal
A$. From the fact that ${\mathcal  B}_1, \ldots, {\mathcal  B}_n$
are mutually commuting subalgebras, it follows that
\begin{eqnarray*}e_{ij}^{(n,s)}&=&\sum_{s_1=1}^{r_1}\sum_{i_1=1}^{k_{s_1}^{1}}\cdots
\sum_{s_{n-1}=1}^{r_{n-1}}\sum_{i_{n-1}=1}^{k_{s_{n-1}}^{(n-1)}}e_{i_{n-1},k_{s_{n-1}}^{(n-1)}}^{(n-1,s_{n-1})}\\
&&\cdots e_{i_1,k_{s_{1}}^{(1)}}^{(1,s_1)}\cdot (p_1\cdots
p_{n-1}e_{ij}^{(n,s)})\cdot e_{k_{s_{1}}^{(1)},i_1}^{(1,s_1)}\cdots
e_{k_{s_{n-1}}^{(n-1)}, i_{n-1}}^{(n-1,s_{n-1})}
\end{eqnarray*}  is in $C^*(a,b)$, which implies that  $\{ e_{ij}^{(n,s)}:1\leq i,j\leq k_s^{(n)}, 1\leq s\leq r_n\}$,
therefore   ${\mathcal  B}_n$, $p_n$ and $z_n$,  are contained in
$C^*(a,b)$. {\it This completes the proof of the claim.}
\vspace{0.3cm}

  By Claim \ref{claim, elements is in it} and Claim \ref{claim, algebra is in it}, $\mathcal
A$ is generated by two self-adjoint elements $a$ and $b$. Therefore
$\mathcal  A$ is singly generated. \hfill
\end{proof}

Suppose $\mathcal A$ is a unital separable C$^*$-algebra and
$\mathcal B$ is a UHF algebra. It is clear that ${\mathcal A}\otimes
{\mathcal B}$ is approximately divisible. Therefore Theorem 9 in
\cite{olsen} is a corollary of our theorem.

\begin{corollary}If $\mathcal A$ is a unital separable C$^*$-algebra and
$\mathcal B$ is a UHF algebra, then ${\mathcal A}\otimes {\mathcal
B}$ is singly generated.
\end{corollary}

\section{Topological free entropy dimension}

In this section we show that the topological free entropy dimension
of any finite generating set of a  unital separable approximately
divisible C$^*$-algebra is less than or equal to $1$.

\subsection{Preliminaries}
\hspace{1.5em}
 We are going to recall Voiculescu's definition of
topological free entropy dimension of an $n$-tuple of self-adjoint
elements in a unital C$^*$-algebra.

For any element $(A_1,\ldots,A_n)$ in ${\mathcal  M}_k({\Bbb C})^n$,
define the operator norm on ${\mathcal  M}_k({\Bbb C})^n$ by
$$\|(A_1,\ldots,A_n)\|=\max\{\|A_1\|,\ldots,\|A_n\|\}.$$

For every $\omega>0$, we define the $\omega$-$\|\cdot\|$-ball
$Ball(B_1,\ldots,B_n;\omega, \|\cdot\|)$ centered at
$(B_1,\ldots,B_n)$ in ${\mathcal  M}_k({\Bbb C})^n$ to be the subset of
${\mathcal  M}_k({\Bbb C})^n$ consisting of all $(A_1,\ldots, A_n)$ in
${\mathcal  M}_k({\Bbb C})^n$ such that
$$\|(A_1,\ldots,A_n)-(B_1,\ldots,B_n)\|< \omega.$$

Suppose $\mathcal  F$ is a subset of ${\mathcal  M}_k({\Bbb C})^n$. We define
the {\it covering number} $\nu_{\infty}({\mathcal  F},\omega)$ to be the
minimal number of $\omega$-$\|\cdot\|$-balls whose union covers $\mathcal  F$ in ${\mathcal  M}_k({\Bbb C})^n$.

 Define ${\Bbb C}\langle
X_1,\ldots, X_n\rangle$ to be the unital noncommutative polynomials
in the indeterminates $X_1,\ldots, X_n$. Let
$\{p_m\}_{m=1}^{\infty}$ be the collection of all noncommutative
polynomials in ${\Bbb C}\langle X_1,\ldots, X_n\rangle$ with
rational complex coefficients. (Here ``rational complex
coefficients'' means that the real and imaginary parts of all
coefficients of $p_m$ are rational numbers).

Suppose $\mathcal{A}$ is a unital C$^{\ast}$-algebra, $x_{1},\ldots
,x_{n},y_{1},\ldots,y_{t}$ are self-adjoint elements of
$\mathcal{A}$. For any $\omega, \varepsilon>0$, positive integers
$k$ and $m$, define
\begin{eqnarray*}
&& \Gamma_{top}(x_{1},\ldots,x_{n};k,\varepsilon, m)=\{(A_{1},\ldots
,A_{n})\in\left({\mathcal{M}}_{k}^{s.a}({\Bbb C})\right)^{n}:\\
&& \qquad \qquad \qquad \qquad \qquad |\Vert p_j(A_{1},\ldots,A_{n})\Vert-\Vert
p_j(x_{1},\ldots,x_{n})\Vert |<\varepsilon, \forall 1\leq j\leq m\},
\end{eqnarray*}
and define
$$\nu_{\infty}(\Gamma_{top}(x_{1},\ldots,x_{n};k,\varepsilon,
m),\omega)$$ to be the covering number of the set
$\Gamma_{top}(x_{1},\ldots,x_{n};k,\varepsilon, m
)$ by $\omega$-$\|\cdot\|$-balls in the metric space $({\mathcal{M}}_{k}^{s.a}%
(\mathbb{C))}^{n}$ equipped with operator norm.

Define
$$\delta_{top}(x_{1},\ldots,x_{n};\omega) =\inf_{\varepsilon>0,m\in\mathbb{N}}\limsup_{k\rightarrow\infty}
\frac{\log(\nu_{\infty}(\Gamma_{top}(x_{1},\ldots,x_{n};k,\varepsilon,
m),\omega))}{-k^{2}\log\omega},$$ and $$
\delta_{top}(x_{1},\ldots,x_{n}) =\limsup_{\omega\rightarrow0^{+}}\delta_{top}(x_{1},\ldots,x_{n};\omega).$$

 Define $\Gamma_{top}(x_{1},\ldots,x_{n}:y_{1,}\ldots
,y_{t};k,\varepsilon, m)$ to be the set of $(A_{1},\ldots,A_{n})\in
\left({\mathcal{M}}_{k}^{s.a}({\Bbb C})\right)^{n}$ such that there
is $(B_{1},\ldots,B_{t})\in \left({\mathcal{M}}_{k}^{s.a}({\Bbb
C})\right)^{t}$ satisfying
$$(A_1,\ldots,A_n,B_1,\ldots,B_t)\in
\Gamma_{top}(x_{1},\ldots,x_{n},y_{1,}\ldots ,y_{t};k,\varepsilon,
m).$$ Then similarly we can define
\begin{eqnarray*}  \delta_{top}(x_{1},\ldots,x_{n}  :y_1,\ldots,y_t;\omega)   = \inf_{\varepsilon>0,m\in\mathbb{N}}\limsup_{k\rightarrow\infty}
\frac{\log(\nu_{\infty}(\Gamma_{top}(x_{1},\ldots,x_{n}:y_1,\ldots,y_t;k,\varepsilon,
m),\omega))}{-k^{2}\log\omega};
\end{eqnarray*} and
$$\delta_{top}(x_{1},\ldots,x_{n}:y_1,\ldots,y_t)
=\limsup_{\omega\rightarrow
0^{+}}\delta_{top}(x_{1},\ldots,x_{n}:y_1,\ldots,y_t;\omega).$$

\begin{lemma}\label{lemma, x:y=x:y,z}Suppose $\mathcal  A$ is a unital C$^*$-algebra, $x_1,\ldots,x_n$, $y_1, \ldots, y_t$ are self-adjoint elements in
$\mathcal  A$ and $x_1,\ldots,x_n$ generate $\mathcal  A$. Suppose $p\in\{p_m\}_{m=1}^{\infty}$ and $\omega>0$. Then the
following are true:

(1)\
$\delta_{top}(x_1,\ldots,x_n;\omega)=\delta_{top}(x_1,\ldots,x_n:y_1,\ldots,y_t;\omega)$,

(2)\ $\delta_{top}(p(x_1,\ldots, x_n):
x_1,\ldots,x_n;\omega)=\delta_{top}(p(x_1,\ldots, x_n):
x_1,\ldots,x_n, y_1,\ldots,y_t;\omega)$,

 (3)\ $\delta_{top}(x_1,\ldots,x_n)\geq\delta_{top}(p(x_1,\ldots, x_n):
x_1,\ldots,x_n).$

\end{lemma}

Proof. The proof of (1) and (2) are straightforward adaptations of
the proof of Proposition 1.6 in \cite{Vo2}. (3) is proved by D.
Hadwin and J. Shen in \cite{Don-shen 2}. \hfill $\Box$

\vspace{0.2cm}

The following lemma is Lemma 2.3 in \cite{b}, and it will be used in
the proofs of Theorem \ref{theorem, upper bound} and Theorem
\ref{theorem, lower bound}.
\begin{lemma}\label{lemma,almost matrix unit is matrix unit}Let ${\mathcal  B}$ be a
finite-dimensional C$^*$-algebra, which is isomorphic to ${\mathcal
M}_{k_1}(\Bbb C)\oplus \cdots \oplus {\mathcal  M}_{k_r}(\Bbb C)$.
For any $\varepsilon>0$, there is a $\delta>0$ such that, whenever
$\mathcal A$ is a unital separable C$^*$-algebra with the unit
$I_{\mathcal  A}$ and $\{a_{ij}^{(s)}:1\leq i,j\leq k_s, 1\leq s\leq
r\}$ in $\mathcal  A$ satisfying

(1)\ \ $\|(a_{ij}^{(s)})^*-a_{ji}^{(s)}\|\leq \delta$ for all
$i,j,s$,

 (2)\ \ $\|\sum_{s=1}^r\sum_{i=1}^{k_s}a_{ii}^{(s)}-I_{\mathcal  A}\|\leq \delta$,

 (3)\ \ $\|a_{ij}^{(s)}a_{jj_1}^{(s)}-a_{ij_1}^{(s)}\|\leq \delta$ for all $i,j,j_1,s$,
 $\|a_{ij}^{(s)}a_{i_1j_1}^{(s_1)}\|\leq\delta$ if $s\neq s_1$ or $j\neq
 i_1$,

 then there is a set $\{e_{ij}^{(s)}:1\leq i,j\leq k_s, 1\leq s\leq r\}$ of matrix units for a copy
of $\mathcal  B$ in $\mathcal  A$ satisfying
$\|a_{ij}^{(s)}-e_{ij}^{(s)}\|<\varepsilon$ for all $i,j,s.$
\end{lemma}

\subsection{Upper bound of topological free entropy dimension in an approximately divisible C$^*$-algebra}
\hspace{1.5em}

 The following lemma is Lemma 6 in \cite{Don}.
\begin{lemma}\label{lemma, covering number}The following statements are true:

(1) Let ${\mathcal  U}_k$ be the group of all unitary matrices in ${\mathcal
M}_k({\Bbb C})$, $\omega>0.$ Then
$$(\frac{1}{\omega})^{k^2}\leq\nu_{\infty}({\mathcal  U}_k,\omega)\leq (\frac{9\pi e}{\omega})^{k^2}.$$

(2) If $d$ is a metric on ${\Bbb R}^m$, $\Bbb B$ is the unit ball of
${\Bbb R}^m$ equipped with the norm induced by $d$, then for
$\omega>0$,
$$(\frac{1}{\omega})^{m}\leq\nu_{d}({\Bbb B}, \omega)\leq
(\frac{3}{\omega})^{m}$$
\end{lemma}

\vspace{0.3cm}

Let $\mathcal  B$ be a finite-dimensional C$^*$-algebra which is
isomorphic to ${\mathcal  M}_{k_1}(\Bbb C)\oplus\cdots
\oplus{\mathcal M}_{k_r}(\Bbb C)$ for some positive integers $k_1,
\ldots, k_r$. To simplify the notation, we will use
$\{e_{st}^{(\iota)}\}_{s,t,\iota}$ to denote a set
$\{e_{st}^{(\iota)}: 1\leq s,t\leq k_{\iota}, 1\leq \iota\leq r\}$
of matrix units for $\mathcal B$, let
$\{\mbox{Re}(e_{st}^{(\iota)})\}_{s,t,\iota}$ denote the set
$\{\frac {\ e_{st}^{(\iota)}+( e_{st}^{(\iota)})^*} 2 :1\leq s,t\leq
k_{\iota},1\leq \iota\leq r\}$, and let
$\{\mbox{Im}(e_{st}^{(\iota)})\}_{s,t,\iota}$ denote the set
$\{\frac {\ e_{st}^{(\iota)}-( e_{st}^{(\iota)})^*}
{2\sqrt{-1}}:1\leq s,t\leq k_{\iota},1\leq \iota\leq r\}$.

\begin{lemma}\label{lemma, less than omega} Let $\mathcal  A$ be a unital separable approximately divisible
C$^*$-algebra with unit $I_{\mathcal  A}$, and $\{x_1,\ldots, x_n\}$
be a family of self-adjoint generators of $\mathcal  A$. Then, for
any $\omega>0$ and positive integer $N$, there exists a
finite-dimensional C$^*$-subalgebra ${\mathcal  B}\subseteq
{\mathcal  A}$ with a set of matrix units
$\{e_{st}^{(\iota)}\}_{s,t,\iota}=\{{e_{st}^{(\iota)}}: 1\leq
s,t\leq k_{\iota}, 1\leq \iota\leq r\}$, a positive integer $m_0$
and $1>\varepsilon_0>0$, such that

(1) $I_{\mathcal  A}\in {\mathcal  B}$,

(2) $\mbox{SubRank}({\mathcal  B})\geq N$,

(3) for any $m\geq m_0$, $\varepsilon\leq \varepsilon_0$, and any
$k\geq 1$, if
$$(A_{1},\ldots,A_{n},\{B_{st}^{(\iota)}\}_{s,t,\iota},\{C_{st}^{(\iota)}\}_{s,t,\iota})
\in\Gamma_{top}(x_{1},\ldots,x_{n},\{\mbox{Re}(e_{st}^{(\iota)})\}_{s,t,\iota},
\{\mbox{Im}(e_{st}^{(\iota)})\}_{s,t,\iota};k,\varepsilon, m),$$
then there exists a set $\{P_{st}^{(\iota)}:1\leq s,t\leq k_{\iota},
1\leq \iota\leq r\}$ of matrix units for a copy of $\mathcal  B$ in ${\mathcal
M}_k(\Bbb C)$ so that
$$\| A_{j}-\sum_{1\leq \iota\leq r}\sum_{1\leq s\leq
k_{\iota}}P_{ss}^{(\iota)}A_{j}P_{ss}^{(\iota)}\|\leq 2\omega.$$

\end{lemma}
\begin{proof} Suppose ${\mathcal  A}=\overline{\cup_m{\mathcal  A}_m}^{\|\cdot\|}$, where ${\mathcal
A}_m$
 is as in Proposition \ref{proposition, ad form}. For any
$\omega>0$, any positive integer $N$ and self-adjoint elements $x_1,\ldots,x_n$, there are self-adjoint
elements $y_{1},\ldots,y_{n}$ in some ${\mathcal{A}}_{m}$ such that
$\Vert x_{j}-y_{j}\Vert<$ $\frac{\omega}{2}$ for all $1\leq j\leq
n$. From part (3) of Proposition \ref{proposition, ad form}, there exists a
finite-dimensional subalgebra $\mathcal{B}$ of
${\mathcal{A}}_{m}'\cap{\mathcal{A}}$ such that $I_{\mathcal  A}\in {\mathcal
B}$ and $\mbox{SubRank}({\mathcal{B}})\geq N$. Let
$\{{e_{st}^{(\iota)}}: 1\leq s,t\leq k_{\iota}, 1\leq \iota\leq r\}$
be a set of matrix units for $\mathcal  B$. Then, for $1\leq j\leq n$,
 \begin{eqnarray*} && \| x_{j}-\sum_{1\leq \iota\leq
r}\sum_{1\leq s\leq k_{\iota}}e_{ss}^{(\iota)}
x_{j}e_{ss}^{(\iota)}\|\\
&=&\|(x_{j}-y_{j})+y_{j}-\sum_{1\leq \iota\leq r}\sum_{1\leq s\leq
k_{\iota} }e_{ss}^{(\iota)}(x_{j}-y_{j})e_{ss}^{(\iota)}-\sum_{1\leq
\iota\leq r}\sum_{1\leq s\leq
k_{\iota}}e_{ss}^{(\iota)}y_{j} e_{ss}^{(\iota)}\|\\
&=& \|(x_{j}-y_{j})-\sum_{1\leq \iota\leq r}\sum_{1\leq s\leq
k_{\iota}}e_{ss} ^{(\iota)}(x_{j}-y_{j})e_{ss}^{(\iota)}\|\\
&\leq &  \|(x_{j}-y_{j})||+||\sum_{1\leq \iota\leq
r}\sum_{1\leq s\leq k_{\iota}}e_{ss}^{(\iota)}(x_{j}-y_{j})e_{ss}^{(\iota)}\|\\
&\leq &
\|(x_{j}-y_{j})||+\max\{||e_{ss}^{(\iota)}(x_{j}-y_{j})e_{ss}^{(\iota)}\|\}\\
&<& \frac{\omega}{2}+\frac{\omega}{2}=\omega.
\end{eqnarray*}

Let $R=\max\{\|x_1\|,\ldots, \|x_n\|,1\}$.

By Lemma \ref{lemma,almost matrix unit is matrix unit}, there are
$0<\varepsilon_0<\min\{1,\frac{\omega}{2}\}$ and positive integer
$m_0$, such that, for any $m\geq m_0$, $\varepsilon\leq
\varepsilon_0$ and $k\geq 1$, if
$$ (A_{1},\ldots,A_{n},\{B_{st}^{(\iota)}\}_{s,t,\iota},
\{C_{st}^{(\iota)}\}_{s,t,\iota})\in\Gamma_{top}(x_{1},\ldots,x_{n},\{\mbox{Re}(e_{st}^{(\iota)})\}_{s,t,\iota},
\{\mbox{Im}(e_{st}^{(\iota)})\}_{s,t,\iota};k,\varepsilon,m),$$ then
there exists a set $\{P_{st}^{(\iota)}: 1\leq s,t\leq k_{\iota},
1\leq \iota\leq r\}\subset{\mathcal{M}}_{k}^{s.a}(\mathbb{C)}$ such
that

(a) $\{P_{st}^{(\iota)}:1\leq s,t\leq k_{\iota}, 1\leq \iota\leq
r\}$ is exactly a set of matrix units for a copy of $\mathcal  B$ in
${\mathcal M}_k(\Bbb C)$,

(b) For any $1\leq \iota\leq r,1\leq s,t\leq k_{\iota}$,
$$\Vert P_{st}^{(\iota)}-(B_{st}^{(\iota)}+\sqrt{-1}\cdot C_{st}^{(\iota)})%
\Vert<\frac{\omega}{24R\cdot\mbox{Rank}({\mathcal  B})}.$$

Let $D_{st}^{(\iota)}=B_{st}^{(\iota)}+\sqrt{-1}\cdot
C_{st}^{(\iota)}$. We have
\begin{eqnarray*}&& \| A_{j}-\sum_{1\leq \iota\leq r}\sum_{1\leq
s\leq
k_{\iota}}P_{ss}^{(\iota)}A_{j}P_{ss}^{(\iota)}\|\\
&\leq &   \|A_{j}-\sum_{1\leq \iota\leq r}\sum_{1\leq s\leq
k_{\iota}}D_{ss}^{(\iota)}A_{j}D_{ss}^{(\iota)}\|+\|\sum_{1\leq
\iota\leq r}\sum_{1\leq s\leq
k_{\iota}}(P_{ss}^{(\iota)}-D_{ss}^{(\iota)})A_{j}D_{ss}^{(\iota)}\|\\
&&+\|\sum_{1\leq \iota\leq r}\sum_{1\leq s\leq
k_{\iota}}D_{ss}^{(\iota)}A_{j}(P_{ss}^{(\iota)}-D_{ss}^{(\iota)})\|\\
&&+\|\sum_{1\leq \iota\leq r}\sum_{1\leq s\leq
k_{\iota}}(P_{ss}^{(\iota)}-D_{ss}^{(\iota)})A_{j}(P_{ss}^{(\iota)}-D_{ss}^{(\iota)})\|\\
&\leq &
\omega+\varepsilon+\frac{\omega}{6}+\frac{\omega}{6}+\frac{\omega}{6}\\
& \leq &   2\omega.
\end{eqnarray*}
\end{proof}

\begin{theorem}\label{theorem, upper bound}
Suppose $\mathcal{A}$ is a unital separable approximately divisible
C$^{*}$-algebra generated by self-adjoint elements $x_1,\ldots,
x_n$. Then
$$\delta _{top}(x_1,\ldots, x_n)\leq 1.$$
\end{theorem}
\begin{proof}
For any positive integer $N$, $1>\omega>0$, from Lemma \ref{lemma,
less than omega}, there exists a finite-dimensional C$^*$-subalgebra
${\mathcal  B}\subseteq {\mathcal  A}$ with a set of matrix units
$\{e_{st}^{(\iota)}\}_{s,t,\iota}=\{{e_{st}^{(\iota)}}: 1\leq
s,t\leq k_{\iota}, 1\leq \iota\leq r\}$, a positive integer $m_0$
and $1>\varepsilon_0>0$, such that

(a)\ $I_{\mathcal  A}\in {\mathcal  B}$, where $I_{\mathcal A}$ is
the unit of $\mathcal A$,

(b)\ $\mbox{SubRank}({\mathcal  B})\geq N$,

(c)\ for $m\geq m_0$ and $\varepsilon\leq \varepsilon_0$, and for
any $k\geq 1$, if
\begin{equation}(A_{1},\ldots,A_{n},\{B_{st}^{(\iota)}\}_{s,t,\iota},\{C_{st}^{(\iota)}\}_{s,t,\iota})
\in\Gamma_{top}(x_{1},\ldots,x_{n},\{\mbox{Re}(e_{st}^{(\iota)})\}_{s,t,\iota},
\{\mbox{Im}(e_{st}^{(\iota)})\}_{s,t,\iota};k,\varepsilon,
m),\end{equation} then there exists a set $\{P_{st}^{(\iota)}:1\leq
s,t\leq k_{\iota}, 1\leq \iota\leq r\}$ of matrix units for a copy
of $\mathcal B$ in ${\mathcal M}_k(\Bbb C)$ so that
$$\| A_{j}-\sum_{1\leq \iota\leq r}\sum_{1\leq s\leq
k_{\iota}}P_{ss}^{(\iota)}A_{j}P_{ss}^{(\iota)}\|\leq 2\omega.$$

Note that $\{P_{ss}^{(\iota)}:{1\leq \iota\leq r, \ 1\leq s\leq
k_{\iota}}\}$ is a family of mutually orthogonal projections  with
the sum $I_k$ in ${\mathcal{M}}_{k}(\mathbb{C)}$. There is some
unitary matrix  $U\in \mathcal U_k$ such that $U^{\ast
}P_{ss}^{(\iota)}U(=Q_{ss}^{(\iota)})$ is diagonal for any $1\leq
\iota\leq r$ and $1\leq s\leq k_{\iota}$. Then, for any $1\leq j\leq
n$, \begin{equation}\Vert A_{j}-U(\sum_{1\leq \iota\leq
r}\sum_{1\leq s\leq k_{\iota}}Q_{ss}^{(\iota)}(U^{\ast
}A_{j}U)Q_{ss}^{(\iota)})U^{\ast}\Vert\leq 2\omega.\end{equation}
Thus, for $1\leq j\leq n$,
$$\|\sum_{1\leq \iota\leq r}\sum_{1\leq s\leq
k_{\iota}}Q_{ss}^{(\iota)}(U^{\ast }A_{j}U)Q_{ss}^{(\iota)}\|\leq
\|A_j\|+2\omega\leq 4R.$$ Therefore \begin{equation}(\sum_{1\leq
\iota\leq r}\sum_{1\leq s\leq
k_{\iota}}Q_{ss}^{(\iota)}(U^*A_{1}U)Q_{ss}^{(\iota)},\ldots,
\sum_{1\leq \iota\leq r}\sum_{1\leq s\leq
k_{\iota}}Q_{ss}^{(\iota)}(U^*A_{n}U)Q_{ss}^{(\iota)})\in
Ball(0,\ldots,0; 4R,\|\cdot\|), \end{equation} i.e., it is contained
in the
 ball centered at $(0,\ldots, 0)$ with radius $4R$ in $(\mathcal M_k(\Bbb C))^n$.

Since $\{P_{st}^{(\iota)}:1\leq s,t\leq k_{\iota}, 1\leq \iota\leq
r\}$ is a system of  matrix units for a copy of  $\mathcal  B$ in
${\mathcal M}_k(\Bbb C)$ such that $$ \sum_{1\leq \iota\leq r, \
1\leq s\leq k_{\iota}} P_{ss}^{(\iota)}=I_k,
$$ we know that there is a unital embedding from $\mathcal B$ into $\mathcal M_k(\Bbb C)$. It follows that
there are positive integers $c_1,\ldots,c_r$ satisfying
\begin{enumerate}
  \item [(i)] $Rank \ P_{11}^{(\iota)}=\cdots
=Rank \ P_{k_{\iota},k_{\iota}}^{(\iota)}=c_\iota$ for all $1\le \iota\le r$, where Rank $T$
is the rank of the matrix $T$ for any $T$ in $\mathcal M_k(\Bbb C)$; and
  \item [(ii)] $c_{1}k_{1}+\cdots+c_{r}k_{r}=k$.
\end{enumerate}
By the restriction on the C$^*$-algebra $\mathcal B$ (see condition
(b) as above), we know that $SubRank(B)\ge N$, i.e.,
$$
\min\{k_1,\ldots, k_r\}\ge N.
$$ By (ii), we obtain that \begin{equation}
\min\{c_1,\ldots, c_r\}\le \frac{k}{N}.\end{equation}
By (i), we know that
$$
Rank \ Q_{11}^{(\iota)}=\cdots =Rank \
Q_{k_{\iota},k_{\iota}}^{(\iota)}=Rank \ P_{11}^{(\iota)}=\cdots
=Rank \ P_{k_{\iota},k_{\iota}}^{(\iota)}=c_\iota, \ \ \ for \ 1\le
\iota\le r.
$$
Thus
the real-dimension
of the linear space $\sum_{1\leq \iota\leq r}\sum_{1\leq j\leq k_{\iota}}Q_{jj}^{(\iota)}%
{\mathcal M}{_{k}(\mathbb{C)}}^{s.a}Q_{jj}^{(\iota)}$ is \begin{equation}
dim_{\Bbb R}\left (\sum_{1\leq \iota\leq r}\sum_{1\leq j\leq k_{\iota}}Q_{jj}^{(\iota)}%
{\mathcal M}{_{k}(\mathbb{C)}}^{s.a}Q_{jj}^{(\iota)}\right )= c_{1}^{2}k_{1}+\cdots+c_{r}%
^{2}k_{r}.\end{equation} By the inequality (16), we get
\begin{equation}c_{1}^{2}
k_{1}+\cdots+c_{r}^{2}k_{r}\leq\frac{k}{N}(c_{1}k_{1}+\cdots+c_{r}k_{r}
)=\frac{k^{2}}{N}.\end{equation}

For any such family of positive integers $c_1,\ldots,c_r$ with
$c_1k_{1}+\cdots+c_{r}k_{r}=k$, and the family of mutually
orthogonal diagonal projections $\{Q_{ss}^{(\iota)}\}_{1\leq s\leq
k_{\iota}, 1\leq \iota\leq r}$ with
$$
Rank(Q_{ss}^{(\iota)})=c_\iota, \ \ \forall \ 1\leq \iota\leq r,
$$  we define
\begin{eqnarray*}
\Omega(\{Q_{ss}^{(\iota)}\}_{  s ,
  \iota })   &=&  \{  (\sum_{1\leq \iota\leq
r}\sum_{1\leq s\leq
k_{\iota}}Q_{ss}^{(\iota)}T_1Q_{ss}^{(\iota)},\ldots,
\sum_{1\leq \iota\leq r}\sum_{1\leq s\leq
k_{\iota}}Q_{ss}^{(\iota)}T_nQ_{ss}^{(\iota)}): \\
&&  T_i=T_i^*\in \mathcal M_k(\Bbb C), \ \forall \ 1\le i\le n \},
\end{eqnarray*} which is a subset of $(\mathcal M_k^{s.a.}(\Bbb
C))^n$.

Therefore, combining Lemma \ref{lemma, covering
number}, equation (17) and inequality (18), for any $\omega>0$ we have
\begin{equation}\nu_{\infty}(\Omega(\{Q_{ss}^{(\iota)}\}_{  s ,
  \iota }) \ \cap \ Ball (0,\ldots,0;4R, \|\ \|); \ \omega
)\leq(\frac{12R}{\omega})^{\frac{nk^{2}}{N}}.\end{equation}
Let
$$
\Lambda = \{(c_1,\ldots,c_r) : \exists k_1,\ldots,k_r\in {\Bbb N}\
\mbox{such that}\ c_1k_1+\cdots +c_rk_r=k\}.$$ By inequality (13),
we know that the cardinality of the set $\Lambda$ satisfies
\begin{equation}Card(\Lambda) \le (\frac{k}{N})^{r}.\end{equation}
Let
\begin{eqnarray*}
\Omega&=& \cup_{(c_1,\ldots,c_r)\in\Lambda} \left \{
\Omega(\{Q_{ss}^{(\iota)}\}_{  s ,
  \iota })  \ | \    \{Q_{ss}^{(\iota)}\}_{1\leq s\leq k_{\iota},
1\leq \iota\leq r} \mbox{ is a family of mutually }\right . \\
&&\left . \mbox{ orthogonal   diagonal projections with $
Rank(Q_{ss}^{(\iota)})=c_\iota,$} \ \ \forall \ 1\leq \iota\leq r
  \right \}.
\end{eqnarray*}
By inequalities (19) and (20), we know that
\begin{equation}
\nu_{\infty}(\Omega  \ \cap \ Ball (0,\ldots,0;4R, \|\ \|); \ \omega
) \le  (\frac{12R}{\omega})^{\frac{nk^{2}}{N}}\cdot  (\frac{k}{N})^{r}.
\end{equation}

Based on (13), (14), (15), (21), and Lemma \ref{lemma, covering
number}, now it is a standard argument to show:
\begin{eqnarray*}
 &&
\delta_{top}(x_{1},\ldots,x_{n}:\{\mbox{Re}(e_{st}^{(\iota)})\}_{s,t,\iota},
\{\mbox{Im}(e_{st}^{(\iota)})\}_{s,t,\iota};4\omega) \\
&  \leq & \limsup_{k\rightarrow\infty}\frac{\log
\left((\frac{k}{N})^{r}(\frac{12R}{\omega})^{\frac{nk^{2}}{N}}(\frac{9\pi
e}{\omega})^{k^{2}}\right)}{-k^2\log (4\omega)}\\
&=& 1+\frac{n}{N}+\frac{\frac{n}{N}\log (12R)+\log(9\pi
e)}{-\log\omega}.
\end{eqnarray*}
By Lemma \ref{lemma, x:y=x:y,z},
$$\delta_{top}(x_{1},\ldots,x_{n};4\omega)\leq
1+\frac{n}{N}+\frac{\frac{n}{N}\log (12R)+\log(9\pi
e)}{-\log\omega}.$$ Therefore $$\delta_{top}(x_1,\ldots,
x_n)=\limsup_{\omega\rightarrow0^{+}}\delta_{top}(x_1,\ldots,x_n;4\omega)\leq
1+\frac{n}{N}.$$
\bigskip Since $N$ is arbitrarily large, $\delta_{top}(x_{1},\ldots
,x_{n})\leq 1.$ \end{proof}

\subsection{Lower bound of topological free entropy dimension of an approximately divisible C$^*$-algebra}

\hspace{1.5em}

The following definition is Definition 5.3 in
\cite{Don-shen 2}.

\begin{definition} Suppose $\mathcal  A$ is a unital C$^*$-algebra and $x_1,\ldots,
x_n$ is a family of self-adjoint elements of $\mathcal  A$ that
generates $\mathcal  A$ as a C$^*$-algebra. If for any $m\in\Bbb N$,
$\varepsilon>0$, there is a sequence of positive integers
$k_1<k_2<\cdots$ such that, for $s\geq 1$,
$$\Gamma_{top}(x_1,\ldots, x_n:y_1,\ldots, y_t;k_s,\varepsilon, m)\neq \emptyset,$$
then $\mathcal  A$ is called having approximation property.
\end{definition}

Using the idea in the proof of Lemma 5.4 in \cite{Don-shen 2}, we
can prove the following lemma.

\begin{lemma}\label{estimation,top}Suppose
$m_1,m_2,\ldots, m_r$ is a family of positive integers with
summation $m$ and $m_1,\ldots, m_r\geq N$ for some positive integer
$N$. Suppose $k_1, \ldots, k_{m}$ is a family of positive integers
with summation $k$ and for every $1\leq s\leq r$, $k_{m_1+\cdots
+m_{s-1}+1}=\cdots=k_{m_1+\cdots +m_{s}}$ ($m_0=0$). If $A=A^*\in
{\mathcal  M}_k({\Bbb C})$, and for some $U\in {\mathcal  U}_k$,
$$\|A-U\left(\begin{array}{llll}1\cdot I_{k_1}&0&\cdots &0\\
0& 2\cdot I_{k_2}& \cdots &0\\
\cdots &\cdots &\ddots&\cdots\\
0&0& \cdots &m\cdot I_{k_m}\end{array}\right)U^*\|\leq
\frac{2}{N^3},$$ then, for any $\omega>0$, we have
$$\nu_{\infty}(\Omega(A), \omega)\geq
(8C_1\omega)^{-k^2}\left(\frac{2C}{\omega}\right)^{\frac{-50k^2}{N}},$$
for some constants $C_1, C>1$ independent of $k,\omega$, where
$$\Omega(A)=\{W^*AW: W\in {\mathcal  U}_k\}.$$
\end{lemma}

Now we are ready to prove the main theorem in this subsection.
\begin{theorem}\label{theorem, lower bound}
Let $\mathcal{A}$ be a unital separable approximately divisible
C$^{*}$-algebra generated by self-adjoint elements $x_1,\ldots,
x_n$. If $\mathcal  A$ has approximation property, then $$\delta
_{top}(x_1,\ldots, x_n)\geq 1.$$
\end{theorem}

\begin{proof} For any positive integer $N$, by part (3) of Proposition
\ref{proposition, ad form}, there is a finite-dimensional
C$^*$-subalgebra $\mathcal  B$ containing the unit of $\mathcal  A$
with $\mbox{SubRank}({\mathcal  B})\geq N$. Therefore there are
positive integers $r, k_1,\ldots, k_r$ such that
$$
\mathcal B\simeq \mathcal M_{k_1}(\Bbb C) \oplus \cdots \oplus \mathcal M_{k_r}(\Bbb C).
$$
Let $\{e_{st}^{(\iota)}: 1\leq
\iota\leq r, 1\leq s,t\leq k_{\iota}\}$ be a system of matrix units for
$\mathcal  B$. Let
$$z_N=\sum_{\iota=1}^{r}\sum_{s=1}^{k_{\iota}}\left (s+\sum_{j=1}^{\iota-1}k_j \right )\cdot
e_{ss}^{(\iota)}.$$  Note that $\{p_m(x_1, \ldots,
x_n)\}_{m=1}^{\infty}$ is a norm-dense set in $\mathcal  A$. There exists a
polynomial $p_{m_N}\in\{p_m\}_{m=1}^{\infty}$ such that
$p_{m_N}(x_1,\ldots, x_n)$ is self-adjoint and
$\|p_{m_N}(x_1,\ldots, x_n)-z_N\|\leq \frac{1}{N^3}$.

For sufficiently small $\varepsilon>0$, sufficiently large positive
integers $m$ and $k$, if $$\begin{aligned} & (B,A_1,\ldots,
A_n,\{C_{st}^{(\iota)}\}_{s,t,\iota},
\{D_{st}^{(\iota)}\}_{s,t,\iota})\\&\qquad \qquad \qquad\in
 \Gamma_{top}(p_{m_N}(x_1,\ldots,x_n),
x_1,\ldots,x_n,\{\mbox{Re}(e_{st}^{(\iota)})\}_{s,t,\iota},\{\mbox{Im}(e_{st}^{(\iota)})\}_{s,t,\iota};k,\varepsilon,
m),\end{aligned}$$ then, by Lemma \ref{lemma,almost matrix unit is
matrix unit}, there exists a set $\{P_{st}^{(\iota)}:1\leq s,t\leq
k_{\iota}, 1\leq \iota\leq r\}$ of matrix units for a copy of
$\mathcal B$ in ${\mathcal  M}_k({\Bbb C})$, such that
$$\|B-\sum_{\iota=1}^{r}\sum_{s=1}^{k_{\iota}}\left (s+\sum_{j=1}^{\iota-1}k_j \right )\cdot
P_{ss}^{(\iota)}\|\leq \frac{2}{N^3}.$$

Let $U$ be a unitary matrix in ${\mathcal  M}_k({\Bbb C})$ such that, for
any $1\leq s\leq k_{\iota}$ and $1\leq \iota\leq r$,
$U^*P_{ss}^{(\iota)}U(=Q_{ss}^{(\iota)})$ is diagonal. Then, from
the preceding inequality,
$$\|B-U\left(\sum_{\iota=1}^{r}\sum_{s=1}^{k_{\iota}}\left (s+\sum_{j=1}^{\iota-1}k_j \right )\cdot
Q_{ss}^{(\iota)}\right)U^*\|\leq \frac{2}{N^3} .$$

From Lemma \ref{estimation,top}, for any $\omega>0$, when $m$ is
large enough and $\varepsilon$ is small enough, there are some
constants $C, C_1>1$ independent of $k$ and $\omega$, such that
$$\begin{aligned} &\nu_{\infty}(\Gamma_{top}(p_{m_N}(x_1,\ldots,x_n):x_1,\ldots,x_n,
\{\mbox{Re}(e_{st}^{(\iota)})\}_{s,t,\iota},
\{\mbox{Im}(e_{st}^{(\iota)})\}_{s,t,\iota};k,\varepsilon, m),
\omega)\\
&\qquad \qquad \qquad \qquad \qquad \geq
(8C_1\omega)^{-k^2}\left(\frac{2C}{\omega}\right)^{\frac{-50k^2}{N}}.\end{aligned}$$
Therefore
$$\delta_{top}(p_{m_N}(x_1,\ldots,x_n):x_1,\ldots, x_n,\{\mbox{Re}(e_{st}^{(\iota)})\}_{s,t,\iota},
\{\mbox{Im}(e_{st}^{(\iota)})\}_{s,t,\iota})\geq 1-\frac{50}{N}.$$
By Lemma \ref{lemma, x:y=x:y,z},
\begin{eqnarray*}&&\delta_{top}(p_{m_N}(x_1,\ldots, x_n): x_1,\ldots,
x_n,\{\mbox{Re}(e_{st}^{(\iota)})\}_{s,t,\iota},\{\mbox{Im}(e_{st}^{(\iota)})\}_{s,t,\iota}
)\\
& =& \delta_{top}(p_{m_N}(x_1,\ldots,x_n): x_1,\ldots, x_n)\\
&\leq & \delta_{top}(x_1,\ldots, x_n),\end{eqnarray*} whence
$\delta_{top}(x_1,\ldots,x_n)\geq 1-\frac{50}{N}$.
 Since $N$ is
an arbitrary positive integer, we obtain
$$\delta_{top}(x_1,\ldots,x_n)\geq 1.$$\end{proof}

Combining Theorem 3.1, Theorem 4.1 and Theorem 4.2, we have the following result.
\begin{theorem}\label{theorem, lower bound}
Let $\mathcal{A}$ be a unital separable approximately divisible
C$^{*}$-algebra. If $\mathcal  A$ has approximation property, then $$\delta
_{top}(x_1,\ldots, x_n)= 1,$$ where $x_1,\ldots,x_n$ is any family of self-adjoint generators of $\mathcal A$.
\end{theorem}

Using the similar idea in the proof of Theorem \ref{theorem, lower
bound}, we can have the following generalized theorem.

\begin{theorem} Let $\mathcal  A$ be a unital separable C$^*$-algebra generated by self-adjoint elements $x_1,\ldots, x_n$. Suppose,
for any positive integer $N$, there is a finite-dimensional
subalgebra in $\mathcal  A$ containing the unit of $\mathcal  A$ with subrank
at lest $N$. If $\mathcal  A$ has approximation property, then
$\delta_{top}(x_1,\ldots, x_n)= 1$.
\end{theorem}

\section{Similarity degree}

In 1955, R. Kadison \cite{Kadison1} formulated the following
conjecture: Let $\mathcal  A$ be a unital C$^*$-algebra and let
$\pi: {\mathcal  A}\rightarrow {\mathcal  B(\mathcal H)}$ ($\mathcal
H$ is a Hilbert space) be a unital bounded homomorphism. Then $\pi$
is similar to a
*-homomorphism, that is, there exists an invertible operator
$S\in{\mathcal  B(\mathcal H)}$ such that $S^{-1}\pi(\cdot)S$ is a
*-homomorphism.

This conjecture remains unproved, although many partial results are
known. U. Haagerup \cite{Haagerup 1} proved that $\pi$ is similar to
a *-homomorphism if and only if it is completely bounded. Moreover,
$$\|\pi\|_{cb}=\inf\{\|S\|\cdot \|S^{-1}\|\}$$ where the infimum runs over
all invertible $S$ such that $S^{-1}\pi(\cdot)S$ is a
*-homomorphism. By definition, $\|\pi\|_{cb}=\sup_{n\geq
1}\|\pi_n\|$ where $\pi_n: {\mathcal  M}_n({\mathcal  A})\rightarrow
{\mathcal M}_n({\mathcal B(\mathcal H)})$ is the mapping taking $n$
by $n$ matrix $[a_{ij}]_{n\times n}$ to matrix
$[\pi(a_{ij})]_{n\times n}$.

G. Pisier \cite{Pisier 1} proved that if a unital C$^*$-algebra
$\mathcal  A$ verifies Kadison's conjecture, then there is a number
$d$ for which there exists a constant $K$ so that any bounded
homomorphism $\pi: {\mathcal  A}\rightarrow {\mathcal B(\mathcal
H)}$ satisfies $\|\pi\|_{cb}\leq K\|\pi\|^d$. Moreover, the smallest
number $d$ with the property is an integer denoted by $d({\mathcal
A})$ and called {\it similarity degree}. It is clear that a
C$^*$-algebra $\mathcal  A$ verifies Kadison's conjecture if and
only if $d({\mathcal  A})<\infty.$

\begin{remark}\label{remark, one-to-one}When determining $d({\mathcal  A})$, it is only necessary to
consider unital bounded homomorphisms that are one-to-one. To see
this, let $\pi_0$ be a unital *-isomorphism from $\mathcal  A$ to
$\mathcal B(\mathcal K)$ for some Hilbert space $\mathcal  K$. It is
not difficult to see that $\pi\oplus\pi_0$ is one-to-one,
$\|\pi\oplus \pi_0\|=\|\pi\|$ and $\|\pi\oplus
\pi_0\|_{cb}=\|\pi\|_{cb}.$\end{remark}

We will show that the similarity degree of every unital separable
approximately divisible C$^*$-algebra is at most 5. To do that, we
need the following lemma.

\begin{lemma}\label{lemma,  sd, estimation} Let $\mathcal  A$ be a C$^*$-algebra with the unit $I_{\mathcal  A}$,
 ${\mathcal  A}_0$ and $\mathcal  B$ be commuting C$^*$-subalgebras of $\mathcal  A$ that contain $I_{\mathcal  A}$.
Suppose ${\mathcal  B}\cong {\mathcal  M}_{k_1}(\Bbb C)\oplus\cdots\oplus{\mathcal
M}_{k_r}(\Bbb C)$ with $k_1,\ldots,k_r\geq n\geq 2$ for some
positive integer $n$, and $\{e_{ij}^{(s)}: 1\leq i,j\leq k_s, 1\leq
s\leq r\}$ is a set of matrix units for ${\mathcal  B}$. If $\{a_{ij}:
1\leq i,j\leq n\}\subseteq {\mathcal  A}_0$, then
$$\|\sum_{1\leq s\leq r}\sum_{1\leq i,j\leq n}a_{ij}e_{ij}^{(s)}\|=
\|[a_{ij}]_{n\times n}\|.$$\end{lemma}

\begin{proof} Let $p_1=I_{k_1}\oplus\cdots 0\oplus 0,\ldots,
p_r=0\oplus\cdots\oplus 0\oplus I_{k_r}$ be the projections in $\mathcal
B$, where $I_{k_s}$ is the unit of ${\mathcal  M}_{k_s}({\Bbb C})$
($1\leq s\leq r$). Then it is clear that $p_1+\cdots +p_r=I_{\mathcal
A}$ and for any $1\leq s\leq r$, $1\leq i,j\leq k_s$,
$e_{ij}^{(s)}=p_se_{ij}^{(s)}$.

Define
$$\pi: {\mathcal  M}_{k_1}(p_1{\mathcal
A}_{0})\oplus\cdots\oplus{\mathcal  M}_{k_r}(p_r{\mathcal
A}_{0})\rightarrow C^*({\mathcal  A}_{0}, {\mathcal  B})$$ by
$$\pi([p_1a_{ij}^{(1)}]_{k_1\times k_1}\oplus \cdots \oplus
[p_ra_{ij}^{(r)}]_{k_r\times
k_r})=\sum_{s=1}^r\sum_{i,j=1}^{k_s}a_{ij}^{(s)}e_{ij}^{(s)},$$ for
any $a_{ij}^{(s)}\in {\mathcal A}_0$. It is clear that $\pi$ is a
*-isomorphism.

Thus, in ${\mathcal  M}_n({\mathcal  A})$,
\begin{eqnarray*}\|[a_{ij}]_{n\times n}\|&=&\|\sum_{s=1}^r\left(\begin{array}{lll}p_s&&0\\
&\ddots&\\
0&&p_s\end{array}\right)[a_{ij}]_{n\times n}\|\\
&=&\max\{\|[p_sa_{ij}]_{n\times n}\|: 1\leq s\leq
n\}.\end{eqnarray*}

On the other hand,
\begin{eqnarray*}&&\|\sum_{1\leq s\leq
r}\sum_{1\leq i,j\leq n}a_{ij}e_{ij}^{(s)}\|\\
&=& \|\pi\left(\left(\begin{array}{ll}[p_1a_{ij}]_{n\times n}& 0\\
0&0\end{array}\right)\oplus \cdots \oplus \left(\begin{array}{ll}[p_ra_{ij}]_{n\times n}& 0\\
0&0\end{array}\right)\right)\|\\
&=&\max\{\|[p_sa_{ij}]_{n\times n}\|: 1\leq s\leq r\}.
\end{eqnarray*}\end{proof}

\begin{theorem}\label{theorem, similarity degree} If
$\mathcal  A$ is a unital separable approximately divisible
C$^*$-algebra, then $$d({\mathcal  A})\leq 5.$$
\end{theorem}

\begin{proof} Let ${\mathcal  A}=\overline{\cup_m{\mathcal  A}_m}^{\|\cdot\|}$ with ${\mathcal  A}_m$
defined in Proposition \ref{proposition, ad form}. By Remark
\ref{remark, one-to-one}, let $\pi: {\mathcal  A}\rightarrow
{\mathcal  B(\mathcal H)}$ be a one-to-one unital bounded
homomorphism, where $\mathcal  H$ is a Hilbert space. It is
sufficient to prove that
$$\|\pi|_{\cup_m{\mathcal  A}_m}\|_{cb}\leq K\|\pi\|^5$$ for some
constant $K$.

 For any positive integer $n$,
let $\{a_{ij}: 1\leq i,j\leq n\}$ be a family of elements in
$\cup_m{\mathcal  A}_m$. Then there exists some positive integer
$m_0$ such that $\{a_{ij}: 1\leq i,j\leq n\}$ is in ${\mathcal
A}_{m_0}$. From Proposition \ref{proposition, ad form}, there exists
a finite-dimensional C$^*$-subalgebra ${\mathcal  B}$ containing the
unit of $\mathcal  A$ with $\mbox{SubRank}({\mathcal  B})\geq n$ and
${\mathcal B}\subset {\mathcal  A}_{m_0}'\cap {\mathcal  A}$. Let
$\{e_{ij}^{(s)}: 1\leq i,j\leq k_s, 1\leq s\leq r\}$ be a set of
matrix units for ${\mathcal B}$.

Since $\mathcal  B$ is finite-dimensional, it follows that $\mathcal
B$ is nuclear. Therefore, from \cite{Haagerup 1}, there exists an
invertible operator $S$ in $\mathcal  B(\mathcal H)$, such that
$\|S\|\cdot \|S^{-1}\|\leq C\|\pi\|^2$ for some constant $C$, and
$S^{-1}\pi|_{\mathcal  B}S$ is a $^*$-isomorphism. Let
$\rho=S^{-1}\pi S$. Then $\{\rho(e_{ij}^{(s)}): 1\leq i,j\leq k_s,
1\leq s\leq r\}$ is a set of matrix units for the C$^*$-algebra
$\rho({\mathcal  B})$. Hence, by Lemma \ref{lemma, sd, estimation},
$$\|\rho (\sum_{1\leq s\leq r}\sum_{1\leq i,j\leq n}a_{ij}e_{ij}^{(s)})\|
\leq \|\rho\|\cdot \|\sum_{1\leq s\leq r}\sum_{1\leq i,j\leq
n}a_{ij}e_{ij}^{(s)}\|=\|\rho\|\cdot \|[a_{ij}]_{n\times n}\|.$$

On the other hand, by Lemma \ref{lemma, sd, estimation},
\begin{eqnarray*}&&\|\rho (\sum_{1\leq s\leq r}\sum_{1\leq i,j\leq n
}a_{ij}e_{ij}^{(s)})\|\\
&=& \|\sum_{s=1}^r\sum_{1\leq i,j\leq n}\rho(a_{ij})\rho(e_{ij}^{(s)})\|\\
&=&\|[\rho(a_{ij})]_{n\times n}\|.
\end{eqnarray*}
Therefore we get $$\|[\rho(a_{ij})]_{n\times n}\|\leq \|\rho\|\cdot
\|[a_{ij}]_{n\times n}\|\leq \|S\|\cdot \|S^{-1}\|\cdot \|\pi\|\cdot
\|[a_{ij}]_{n\times n}\|\leq C\|\pi\|^3\|[a_{ij}]_{n\times n}\|,$$
which implies that $\|\rho|_{\cup_m{\mathcal  A}_m}\|_{cb}\leq
C\|\pi\|^3$, then
$$\|\pi|_{\cup_m{\mathcal  A}_m}\|_{cb}=\|S\rho|_{\cup_m{\mathcal  A}_m} S^{-1}\|_{cb}\leq \|S^{-1}\|\cdot \|S\|\cdot
\|\rho|_{\cup_m{\mathcal  A}_m}\|_{cb}\leq C^2\|\pi\|^5.$$\end{proof}

F. Pop \cite{F. Pop} proved that if $\mathcal  A$ is a unital
C$^*$-algebra, $\mathcal  B$ is a unital nuclear C$^*$-algebra and
contains unital matrix algebras of any order, then the similarity
degree of ${\mathcal  A}\otimes {\mathcal  B}$ is at most 5. Here we state our
another result which generalize F. Pop's result.

To prove our result, we need the following lemma (Corollary 2.3 in
\cite{F. Pop}).

\begin{lemma}\label{lemma,florin pop} Let $\mathcal  A$ and $\mathcal  B$ be unital C$^*$-algebras and
$\mathcal  B$ nuclear. If $\pi$ is a unital bounded homomorphism of
${\mathcal  A}\otimes {\mathcal  B}$ such that $\pi|_{\mathcal  A}$ is completely
bounded and $\pi|_{\mathcal  B}$ is *-homomorphism, then $\pi$ is
completely bounded and $\|\pi\|_{cb}\leq \|\pi|_{\mathcal  A}\|_{cb}$.
\end{lemma}

Using Lemma \ref{lemma,florin pop} and the idea in the proof of
Theorem \ref{theorem, similarity degree}, we can get the following
theorem:

\begin{theorem} Let $\mathcal  A$ be a unital nuclear C$^*$-algebra such that for
any positive integer $N$, there is a finite-dimensional subalgebra
in $\mathcal  A$ containing the unit of $\mathcal  A$ with subrank
at least $N$. Then, for any unital C$^*$-algebra $\mathcal  B$,
$d({\mathcal  A}\otimes {\mathcal  B})\leq 5$.
\end{theorem}

\noindent {\bf Acknowledgement} The authors would like to thank
Professor Don Hadwin for many helpful discussions.

\end{document}